\documentclass[psamsfonts,english]{amsart}

    \usepackage[T1]{fontenc}
    \usepackage{babel}
    \usepackage{natbib,amsmath,amssymb,amsthm,url,enumerate,nicefrac}
    \usepackage[mathcal]{eucal}
    \usepackage[all]{xypic}

    \newtheorem{theorem}{Theorem}
    \newtheorem{lemma}[theorem]{Lemma}
    
    \theoremstyle{definition}
    \newtheorem{definition}[theorem]{Definition}
    \theoremstyle{remark}
    \newtheorem{remark}[theorem]{Remark}

    \newcommand{\FF}{\mathbb{F}}
    \newcommand{\QQ}{\mathbb{Q}}
    \newcommand{\ZZ}{\mathbb{Z}}
    \newcommand{\RR}{\mathbb{R}}
    \newcommand{\CC}{\mathbb{C}}
    \newcommand{\PP}{\mathbb{P}}
    \newcommand{\NN}{\mathbb{N}}
    \newcommand{\jac}{\mathcal{J}_{C}}
    \newcommand{\heltal}[1]{\mathfrak{O}_{#1}}
    \newcommand{\frob}{\varphi}
    \newcommand{\fullED}{\varkappa}

    \DeclareMathOperator{\End}{End}
    \DeclareMathOperator{\Div}{Div}
    \DeclareMathOperator{\Mat}{Mat}
    \DeclareMathOperator{\gal}{Gal}
    \DeclareMathOperator{\divisor}{div}

\begin{document}

\title[Non-Cyclic Subgroups of Jacobians of Genus Two Curves with CM]
{Non-Cyclic Subgroups of Jacobians of Genus Two Curves with Complex Multiplication}

\author[C.R. Ravnshøj]{Christian Robenhagen Ravnshøj}

\address{Department of Mathematical Sciences \\
University of Aarhus \\
Ny Munkegade \\
Building 1530 \\
DK-8000 Aarhus C}

\email{cr@imf.au.dk}

\thanks{Research supported in part by a PhD grant from CRYPTOMAThIC}



% 14H40 Curves, Jacobians
% 14Q05 Computational aspects in algebraic geometry, curves
% 11G15 Arithmetic algebraic geometry (Diophantine geometry), Complex multiplication
% 94A60 Communication, information, Cryptography

\begin{abstract}
Let~$E$ be an elliptic curve defined over a finite field. Balasubramanian and Koblitz have proved that if the
$\ell^\text{th}$ roots of unity $\mu_\ell$ is not contained in the ground field, then a field extension of the ground
field contains $\mu_\ell$ if and only if the $\ell$-torsion points of~$E$ are rational over the same field extension.
We generalize this result to Jacobians of genus two curves with complex multiplication. In particular, we show that the
Weil- and the Tate-pairing on such a Jacobian are non-degenerate over the same field extension of the ground field.
\end{abstract}

\maketitle

\section{Introduction}

In \cite{koblitz87}, Koblitz described how to use elliptic curves to construct a public key cryptosystem. To get a more
general class of curves, and possibly larger group orders, Koblitz \cite{koblitz89} then proposed using Jacobians of
hyperelliptic curves.

In elliptic curve cryptography it is essential to know the number of points on the curve. Cryptographically we are
interested in elliptic curves with large cyclic subgroups. Such elliptic curves can be constructed. The construction is
based on the theory of complex multiplication, studied in detail by \cite{atkin-morain}. It is referred to as the
\emph{CM method}. The CM method for constructing elliptic curves has been generalized to genus~two curves by
\cite{spallek}, and efficient algorithms have been proposed by \cite{weng03} and \cite{gaudry}. Both algorithms take as
input a primitive, quartic CM field $K$ (see section~\ref{sec:CMfields}), and give as output a genus~two curve $C$
defined over a prime field $\FF_p$.

After Boneh and Franklin \cite{boneh-franklin} proposed an identity based cryptosystem by using the Weil pairing on an
elliptic curve, pairings have been of great interest to cryptography~\cite{galbraith05}. The next natural step was to
consider pairings on Jacobians of hyperelliptic curves. Galbraith \emph{et al}~\cite{galbraith07} survey the recent
research on pairings on Jacobians of hyperelliptic curves.

The pairing in question is usually the Weil- or the Tate-pairing; both pairings can be computed with Miller's algorithm
\cite{miller-algorithm}. The Tate-pairing can be computed more efficiently than the Weil-pairing, cf.
\cite{galbraith01}. Let~$C$ be a smooth curve defined over a finite field $\FF_q$, and let~$\jac$ be the Jacobian
of~$C$. Let~$\ell$ be a prime number dividing the number of $\FF_q$-rational points on the Jacobian, and let $k$ be the
multiplicative order of $q$ modulo~$\ell$. By \cite{hess}, the Tate-pairing is non-degenerate
on~$\jac(\FF_{q^k})[\ell]$. By \cite[Proposition~8.1, p.~96]{sil}, the Weil-pairing is non-degenerate on~$\jac[\ell]$.
So if~$\jac[\ell]$ is not contained in~$\jac(\FF_{q^k})$, then the Tate pairing is non-degenerate over a possible
smaller field extension than the Weil-pairing. For elliptic curves, in most cases relevant to cryptography, the
Weil-pairing and the Tate-pairing are non-degenerate over the same field: let~$E$ be an elliptic curve defined
over~$\FF_p$, and consider a prime number~$\ell$ dividing the number of $\FF_p$-rational points on~$E$. Balasubramanian
and Koblitz \cite{balasubramanian} proved that
    \begin{equation}\label{eq:embeddingEllipticCurves}
    \text{\emph{if $\ell\nmid p-1$, then $E[\ell]\subseteq E(\FF_{p^k})$ if and only if $\ell\mid p^k-1$.}}
    \end{equation}
By Rubin and Silverberg \cite{rubin-silverberg07}, this result also holds for Jacobians of genus two curves in the
following sense: \emph{if $\ell\nmid p-1$, then the Weil-pairing is non-degenerate on $U\times V$, where
$U=\jac(\FF_p)[\ell]$, $V=\ker(\frob-p)\cap\jac[\ell]$ and $\frob$ is the $p$-power Frobenius endomorphism on~$\jac$}.

The result~\eqref{eq:embeddingEllipticCurves} can also be stated as: \emph{if $\ell\nmid p-1$, then
$E(\FF_{p^k})[\ell]$ is bicyclic if and only if $\ell\mid p^k-1$}. In this paper, we show that in most cases, this
result also holds for Jacobians of genus two curves with complex multiplication. More precisely, the following theorem
is established.

\setcounter{theorem}{8}

\begin{theorem}%\marginpar{Check: nr.~\ref{teo:main}}
Consider a genus two curve $C$ defined over $\FF_p$ with $\End(\jac)\simeq\heltal{K}$, where $K$ is a primitive,
quartic CM field (cf.~section~\ref{sec:CMfields}). Let $\omega_m$ be a $p^m$-Weil number of the Jacobian~$\jac$. Let
$\ell$ be an odd prime number dividing the number of $\FF_p$-rational points on $\jac$, and with~$\ell$ unramified
in~$K$, $\ell\nmid p$ and $\ell\nmid p-1$. Let $p$ be of multiplicative order $k$ modulo~$\ell$. Then the following
holds.
    \begin{enumerate}[(i)]
    \item If $\omega_m^2\not\equiv 1\pmod{\ell}$, then $\jac(\FF_{p^m})[\ell]$ is bicyclic if and only if~$\ell$ divides
    $p^m-1$.
    \item The Weil-pairing is non-degenerate on $\jac(\FF_{p^k})[\ell]\times\jac(\FF_{p^k})[\ell]$.
    \end{enumerate}
\end{theorem}

\setcounter{theorem}{0}

\subsection*{Notation and assumptions}

In this paper we only consider smooth curves. If $F$ is an algebraic number field, then $\heltal{F}$ denotes the ring
of integers of $F$, and $F_0=F\cap\RR$ denotes the real subfield of $F$.

\section{Genus two curves}\label{sec:HyperellipticCurves}

A hyperelliptic curve is a projective curve $C\subseteq\PP^n$ of genus at least two with a separable, degree two
morphism $\phi:C\to\PP^1$. It is well known, that any genus two curve is hyperelliptic. Throughout this paper, let~$C$
be a curve of genus two defined over a finite field~$\FF_q$ of characteristic~$p$. By the Riemann-Roch Theorem there
exists a birational map \mbox{$\psi:C\to\PP^2$}, mapping~$C$ to a curve given by an equation of the form
    $$y^2+g(x)y=h(x),$$
where $g,h\in\FF_q[x]$ are of degree $\deg(g)\leq 3$ and $\deg(h)\leq 6$; cf.~\cite[chapter~1]{cassels}.

The set of principal divisors $\mathcal{P}(C)$ on~$C$ constitutes a subgroup of the degree zero divisors $\Div_0(C)$.
The Jacobian~$\jac$ of~$C$ is defined as the quotient
    $$\jac=\Div_0(C)/\mathcal{P}(C).$$
{\samepage Let $\ell\neq p$ be a prime number. The $\ell^n$-torsion subgroup~$\jac[\ell^n]\subseteq\jac$ of points of
order dividing $\ell^n$ is a $\ZZ/\ell^n\ZZ$-module of rank four, i.e.
    \begin{equation*}\label{eq:J-struktur}
    \jac[\ell^n]\simeq\ZZ/\ell^n\ZZ\times\ZZ/\ell^n\ZZ\times\ZZ/\ell^n\ZZ\times\ZZ/\ell^n\ZZ;
    \end{equation*}
cf.~\cite[Theorem~6, p.~109]{lang59}.}

The multiplicative order $k$ of $q$ modulo~$\ell$ plays an important role in cryptography, since the (reduced)
Tate-pairing is non-degenerate over $\FF_{q^k}$; cf.~\cite{hess}.

\begin{definition}[Embedding degree]
Consider a prime number $\ell\neq p$ dividing the number of $\FF_q$-rational points on the Jacobian~$\jac$. The
embedding degree of~$\jac(\FF_q)$ with respect to~$\ell$ is the least number $k$, such that $q^k\equiv 1\pmod{\ell}$.
\end{definition}

Closely related to the embedding degree, we have the \emph{full} embedding degree.

\begin{definition}[Full embedding degree]
Consider a prime number $\ell\neq p$ dividing the number of $\FF_q$-rational points on the Jacobian~$\jac$. The full
embedding degree of~$\jac(\FF_q)$ with respect to~$\ell$ is the least number $\fullED$, such that
$\jac[\ell]\subseteq\jac(\FF_{q^\fullED})$.
\end{definition}

\begin{remark}
If~$\jac[\ell]\subseteq\jac(\FF_{q^\fullED})$, then $\ell\mid q^\fullED-1$; cf.~\cite[Corollary~5.77, p.~111]{hhec}.
Hence, the full embedding degree is a multiple of the embedding degree.
\end{remark}

A priori, the Weil-pairing is only non-degenerate over $\FF_{q^\fullED}$. But in fact, as we shall see, the
Weil-pairing is also non-degenerate over $\FF_{q^k}$.

\section{The Weil- and the Tate-pairing}

Let $\FF$ be an algebraic extension of~$\FF_q$. Let $x\in\jac(\FF)[\ell]$ and $y=\sum_ia_i P_i\in\jac(\FF)$ be divisors
with disjoint supports, and let $\bar{y}\in\jac(\FF)/\ell\jac(\FF)$ denote the divisor class containing the
divisor~$y$. Furthermore, let $f_x\in\FF(C)$ be a rational function on~$C$ with divisor $\divisor(f_x)=\ell x$. Set
$f_x(y)=\prod_if(P_i)^{a_i}$. Then $e_\ell(x,\bar{y})=f_x(y)$ is a well-defined pairing
    $$e_\ell:\jac(\FF)[\ell]\times\jac(\FF)/\ell\jac(\FF)\longrightarrow\FF^\times/(\FF^\times)^\ell,$$
it is called the \emph{Tate-pairing}; cf.~\cite{galbraith05}. Raising the result to the
power~$\frac{|\FF^\times|}{\ell}$ gives a well-defined element in the subgroup $\mu_\ell\subseteq\bar\FF$ of the
$\ell^{\mathrm{th}}$ roots of unity. This pairing
    $$\hat{e}_\ell:\jac(\FF)[\ell]\times\jac(\FF)/\ell\jac(\FF)\longrightarrow\mu_\ell$$
is called the \emph{reduced} Tate-pairing. If the field~$\FF$ is finite and contains the $\ell^\mathrm{th}$ roots of
unity, then the Tate-pairing is bilinear and non-degenerate; cf.~\cite{hess}.

Now let $x,y\in\jac[\ell]$ be divisors with disjoint support. The Weil-pairing
    $$e_\ell:\jac[\ell]\times\jac[\ell]\to\mu_\ell$$
is then defined by $e_\ell(x,y)=\frac{\hat{e}_\ell(x,\bar{y})}{\hat{e}_\ell(y,\bar{x})}$. The Weil-pairing is bilinear,
anti-symmetric and non-degenerate on~$\jac[\ell]\times\jac[\ell]$; cf.~\cite{miller}.

\section{Matrix representation of the endomorphism ring}

An endomorphism $\psi:\jac\to\jac$ induces a linear map $\bar{\psi}:\jac[\ell]\to\jac[\ell]$ by restriction. Hence,
$\psi$ is represented by a matrix $M\in\Mat_4(\ZZ/\ell\ZZ)$ on~$\jac[\ell]$. Let $f\in\ZZ[X]$ be the characteristic
polynomial of $\psi$ (see~\cite[pp.~109--110]{lang59}), and let $\bar{f}\in(\ZZ/\ell\ZZ)[X]$ be the characteristic
polynomial of $\bar{\psi}$. Then $f$ is a monic polynomial of degree four, and by \cite[Theorem~3, p.~186]{lang59},
    \begin{equation*}\label{eq:KarPolKongruens}
    f(X)\equiv\bar{f}(X)\pmod{\ell}.
    \end{equation*}

Since~$C$ is defined over~$\FF_q$, the mapping $(x,y)\mapsto (x^q,y^q)$ is a morphism on~$C$. This morphism induces the
$q$-power Frobenius endo\-morphism $\frob$ on the Jacobian~$\jac$. Let $P(X)$ be the characteristic polynomial
of~$\frob$. $P(X)$ is called the \emph{Weil polynomial} of~$\jac$, and
    $$|\jac(\FF_q)|=P(1)$$
by the definition of $P(X)$ (see~\cite[pp.~109--110]{lang59}); i.e. the number of~$\FF_q$-rational points on the
Jacobian is $P(1)$.

\begin{definition}[Weil number]\label{definition:WeilNumber}
Let notation be as above. Let $P_m(X)$ be the characteristic polynomial of the $q^m$-power Frobenius
endomorphism~$\frob_m$ on~$\jac$. Consider a number $\omega_m\in\CC$ with $P_m(\omega_m)=0$. If $P_m(X)$ is reducible,
assume furthermore that $\omega_m$ and $\frob_m$ are roots of the same irreducible factor of $P_m(X)$. We
identify~$\frob_m$ with~$\omega_m$, and we call~$\omega_m$ a \emph{$q^m$-Weil~number} of~$\jac$.
\end{definition}

\begin{remark}\label{rem:P_reducible}
A $q^m$-Weil~number is not necessarily uniquely determined. In general, $P_m(X)$ is irreducible, in which case~$\jac$
has four $q^m$-Weil~numbers.

Assume $P_m(X)$ is reducible. Write $P_m(X)=f(X)g(X)$, where $f,g\in\ZZ[X]$ are of degree at least one. Since
$P_m(\frob_m)=0$, either $f(\frob_m)=0$ or $g(\frob_m)=0$; if not, then either $f(\frob_m)$ or $g(\frob_m)$ has
infinite kernel, i.e. is not an endomorphism of~$\jac$. So a $q^m$-Weil number is well-defined.
\end{remark}

\section{CM fields}\label{sec:CMfields}

An elliptic curve $E$ with $\ZZ\neq\End(E)$ is said to have \emph{complex multiplication}. Let $K$ be an ima\-ginary,
quadratic number field with ring of integers $\heltal{K}$. $K$ is a \emph{CM field}, and if
\mbox{$\End(E)\simeq\heltal{K}$}, then $E$ is said to have \emph{CM by $\heltal{K}$}. More generally a CM field is
defined as follows.

\begin{definition}[CM field]
A number field $K$ is a CM field, if $K$ is a totally imaginary, quadratic extension of a totally real number field
$K_0$.
\end{definition}

In this paper only CM fields of degree $[K:\QQ]=4$ are considered. Such a field is called a \emph{quartic} CM field.

%\begin{remark}\label{rem:quarticCM}
%Consider a quartic CM field $K$. Let $K_0=K\cap\RR$ be the real subfield of $K$. Then $K_0$ is a real, quadratic number
%field, $K_0=\QQ(\sqrt{D})$. By a basic result on quadratic number fields, the ring of integers of $K_0$ is given by
%$\heltal{K_0}=\ZZ+\xi\ZZ$, where
%    $$
%    \xi=\begin{cases}
%    \frac{1+\sqrt{D}}{2}, & \textrm{if $D\equiv 1\pmod{4}$,} \\
%    \sqrt{D}, & \textrm{if $D\not\equiv 1\pmod{4}$}.
%    \end{cases}
%    $$
%Since $K$ is a totally imaginary, quadratic extension of $K_0$, a number $\eta\in K$ exists, such that $K=K_0(\eta)$,
%$\eta^2\in K_0$. The number $\eta$ is totally imaginary, and we may assume that $\eta=i\eta_0$, $\eta_0\in\RR$.
%Furthermore we may assume that $-\eta^2\in\heltal{K_0}$; so $\eta=i\sqrt{a+b\xi}$, where $a,b\in\ZZ$.
%\end{remark}

Let $C$ be a genus two curve. We say that $C$ has CM by~$\heltal{K}$, if $\End(\jac)\simeq\heltal{K}$. The structure of
$K$ determines whether $\jac$ is simple, i.e. does not contains an abelian subvariety other than $\{\mathcal{O}\}$ and
itself. More precisely, the following theorem holds.

\begin{theorem}\label{teo:simple}
Let $C$ be a genus two curve with $\End(\jac)\simeq\heltal{K}$, where $K$ is a quartic CM field. Then $\jac$ is simple
if and only if $K/\QQ$ is Galois with Galois group $\gal(K/\QQ)\simeq\ZZ/2\ZZ\times\ZZ/2\ZZ$.
\end{theorem}

\begin{proof}
\cite[proposition~26, p.~61]{shi}.
\end{proof}

Theorem~\ref{teo:simple} motivates the following definition.

\begin{definition}[Primitive, quartic CM field]\label{def:CMfieldPrimitive}
A quartic CM field $K$ is called primitive if either $K/\QQ$ is not Galois, or $K/\QQ$ is Galois with cyclic Galois
group.
\end{definition}

\section{Non-cyclic subgroups of $\jac$}\label{sec:properties}

Let $K$ be a primitive, quartic CM field. By the CM method (see \cite{weng03,gaudry}), we can construct a genus two
curve $C$ with $\End(\jac)\simeq\heltal{K}$. The following theorem concerns such a curve.

{\samepage
\begin{theorem}\label{teo:main}
Consider a genus two curve $C$ defined over $\FF_p$ with $\End(\jac)\simeq\heltal{K}$, where $K$ is a primitive,
quartic CM field. Let $\omega_m$ be a $p^m$-Weil number of the Jacobian~$\jac$. Let~$\ell$ be an odd prime number
dividing the number of $\FF_p$-rational points on $\jac$, and with~$\ell$ unramified in~$K$, $\ell\nmid p$ and
$\ell\nmid p-1$. Let $p$ be of multiplicative order $k$ modulo~$\ell$. Then the following holds.
    \begin{enumerate}[(i)]
    \item If $\omega_m^2\not\equiv 1\pmod{\ell}$, then $\jac(\FF_{p^m})[\ell]$ is bicyclic if and only if~$\ell$ divides
    $p^m-1$.\label{teo:main:1}
    \item The Weil-pairing is non-degenerate on $\jac(\FF_{p^k})[\ell]\times\jac(\FF_{p^k})[\ell]$.\label{teo:main:2}
    \end{enumerate}
\end{theorem}
}

In the following, let $\bar{P}_m\in(\ZZ/\ell\ZZ)[X]$ be the characteristic polynomial of the restriction of $\frob_m$
to~$\jac[\ell]$. The proof of Theorem~\ref{teo:main} uses a number of lemmas.

\begin{lemma}\label{lem:unramified}
Let notation and assumptions be as in Theorem~\ref{teo:main}. Let $\imath:\heltal{K}\to\End(\jac)$ be an isomorphism.
Consider a number $\alpha\in\heltal{K}$. If $\ker[\ell]\subseteq\ker(\imath(\alpha)^n)$ for some number $n\in\NN$, then
$\ker[\ell]\subseteq\ker(\imath(\alpha))$.
\end{lemma}

\begin{proof}
Since $\ker[\ell]\subseteq\ker(\imath(\alpha)^n)$, it follows that $\imath(\alpha)^n=\ell\tilde\beta$ for some
endomorphism $\tilde\beta\in\End(\jac)$; see e.g. \cite[Remark~7.12, p.~37]{milne:AV}. Notice that
$\tilde\beta=\frac{\imath(\alpha)^n}{\ell}=\imath(\beta)$ for some number $\beta\in\heltal{K}$. Hence,
$\alpha^n=\ell\beta\in\ell\heltal{K}$. Since~$\ell$ is unramified in $K$, it follows that $\alpha\in\ell\heltal{K}$. So
$\ker[\ell]\subseteq\ker(\imath(\alpha))$.
\end{proof}

\begin{lemma}\label{lem:Rank<=2}
Let notation and assumptions be as in Theorem~\ref{teo:main}. If $\omega_m\not\equiv 1\pmod{\ell}$, then
$\jac(\FF_{p^m})[\ell]$ is of rank at most two as a $\ZZ/\ell\ZZ$-module.
\end{lemma}

\begin{proof}
Since $\ell\mid|\jac(\FF_p)|$, $1$ is a root of $\bar{P}_m$. Assume that $1$ is a root of $\bar{P}_m$ of
multiplicity~$\nu$. Since the roots of $\bar{P}_m$ occur in pairs $(\alpha,p^m/\alpha)$, also $p^m$ is a root of
$\bar{P}_m$ of multiplicity~$\nu$.

If~$\jac(\FF_{q^m})[\ell]$ is of rank three as a $\ZZ/\ell\ZZ$-module, then $\ell$ divides $ q^m-1$ by
\cite[Proposition~5.78, p.~111]{hhec}. Choose a basis $\mathcal{B}$ of $\jac[\ell]$. With respect to $\mathcal{B}$,
$\frob_m$ is represented by a matrix of the form
    $$M=\begin{bmatrix}
        1 & 0 & 0 & m_1 \\
        0 & 1 & 0 & m_2 \\
        0 & 0 & 1 & m_3 \\
        0 & 0 & 0 & m_4
        \end{bmatrix}.
    $$
Now, $m_4=\det M\equiv\deg\frob_m=p^{2m}\equiv 1\pmod{\ell}$, so $\bar{P}_m(X)=(X-1)^4$. Since~$\ell$ is unramified in
$K$, it follows that $\omega_m\equiv 1\pmod{\ell}$; cf. Lemma~\ref{lem:unramified}. This is a contradiction. So
$\jac(\FF_{p^m})[\ell]$ is of rank at most two as a $\ZZ/\ell\ZZ$-module.
\end{proof}

\begin{lemma}\label{lem:Pirreducible}
Let notation and assumptions be as in Theorem~\ref{teo:main}. If $\omega_m^2\not\equiv 1\pmod{\ell}$, then $P(X)$ is
irreducible.
\end{lemma}

\begin{proof}
The Jacobian $\jac$ is simple by Theorem~\ref{teo:simple}. Assume $P_m(X)$ is reducible. Then $P_m(X)=f(X)^e$ for
some integer $e\in\ZZ$ and some irreducible polynomial $f\in\ZZ[X]$ by \cite[Theorem~8, p.~58]{milne-waterhouse}.
Notice that $e\in\{2,4\}$. If $\omega_m\notin\RR$, then $\QQ(\omega_m)\subset K$ is an imaginary, quadratic number field
and $K$ is the composition of $K_0$ and $\QQ(\omega_m)$, i.e. $\gal(K/\QQ)$ is bicyclic. This is a contradiction. So
$\omega_m\in\RR$, i.e. $\omega_m^2=p^m$. If $\omega_m\in\QQ$, then $f(X)\equiv X-1\pmod{\ell}$ because $\bar{P}_m(1)=0$.
But then $\omega_m\equiv 1\pmod{\ell}$. This is a contradiction. So $\omega_m\notin\QQ$, $e=2$ and $f(X)=X^2-p^m$.
Hence, $\bar{P}_m(X)=(X^2-p^m)^2$. Since $\bar{P}_m(1)=0$, it follows that $\omega_m^2=p^m\equiv 1\pmod{\ell}$. This is
a contradiction. So $P_m(X)$ is irreducible.
\end{proof}

\begin{proof}[Proof of Theorem~\ref{teo:main}]
Assume that $\jac(\FF_{p^m})[\ell]$ is bicyclic. If $p^m\not\equiv 1\pmod{\ell}$, then $1$ is a root of $\bar{P}_m$ of
multiplicity two, i.e. $\bar{P}_m(X)=(X-1)^2(X-p^m)^2$. $P(X)$ is irreducible by Lemma~\ref{lem:Pirreducible}. Hence,
by \cite[Proposition~8.3, p.~47]{neukirch} it follows that~$\ell$ ramifies in $K$. This is a contradiction. So
$p^m\equiv 1\pmod{\ell}$, i.e. $\ell\mid p^m-1$.

On the other hand, if $\ell\mid p^m-1$, then the Tate pairing is non-degenerate on $\jac(\FF_{p^m})[\ell]$. So
$\jac(\FF_{p^m})[\ell]$ must be of rank at least two as a $\ZZ/\ell\ZZ$-module, since $\ell\nmid p-1$. Hence,
$\jac(\FF_{p^m})[\ell]$ is bicyclic by Lemma~\ref{lem:Rank<=2}. The proof of Theorem~\ref{teo:main},
part~\eqref{teo:main:1} is established.

Now let $m=k$. If $\omega_k\equiv 1\pmod{\ell}$, then $\jac[\ell]=\jac(\FF_{p^k})[\ell]$, and~\eqref{teo:main:2}
follows. Assume that $\omega_k\not\equiv 1\pmod{\ell}$. Let $U=\jac(\FF_p)[\ell]$ and $V=\ker(\frob-p)\cap\jac[\ell]$,
where $\frob$ is the $p$-power Frobenius endomorphism on $\jac$. Then
$V=\jac(\FF_{p^k})[\ell]\setminus\jac(\FF_p)[\ell]$ by Lemma~\ref{lem:Rank<=2}, and
    $$\jac(\FF_{p^k})[\ell]\simeq U\oplus V\simeq\ZZ/\ell\ZZ\times\ZZ/\ell\ZZ.$$
By \cite{rubin-silverberg07}, the Weil-pairing $e_W$ is non-degenerate on $U\times V$. Now let
$x\in\jac(\FF_{p^k})[\ell]$ be an arbitrary $\FF_{p^k}$-rational point of order~$\ell$. Write $x=x_U+x_V$, where
$x_U\in U$ and $x_V\in V$. Choose $y\in V$ and $z\in U$, such that $e_W(x_U,y)\neq 1$ and $e_W(x_V,z)\neq 1$. We may
assume that $e_W(x_U,y)\cdot e_W(x_V,z)\neq 1$; if not, replace $z$ by $2z$. Since the Weil-pairing is anti-symmetric,
$e_W(x_U,z)=e_W(x_V,y)=1$. Hence,
    $$e_W(x,y+z)=e_W(x_U,y)\cdot e_W(x_V,z)\neq 1.$$
\end{proof}

\bibliographystyle{plain}
\bibliography{references}

\end{document}